\documentclass{amsart}
\usepackage{amssymb, amsmath}
\usepackage[latin1]{inputenc}
\usepackage[final]{graphicx}
\newtheorem{theorem}{Theorem}[section]
\newtheorem{lemma}[theorem]{Lemma}
\newtheorem{remark}[theorem]{Remark}
\newtheorem{definition}[theorem]{Definition}

\newtheorem{example}[theorem]{Example}

\makeatletter
 \@addtoreset{equation}{section}
\makeatother

\begin{document}

\title[The quasi-Einstein Equation]{A natural linear equation in affine geometry: The affine quasi-Einstein Equation}
\author{M. Brozos-V\'{a}zquez \, E. Garc\'{i}a-R\'{i}o\, P. Gilkey \, X. Valle-Regueiro}
\address{MBV: Universidade da Coru\~na, Differential Geometry and its Applications Research Group, Escola Polit\'ecnica Superior, 15403 Ferrol,  Spain}
\email{miguel.brozos.vazquez@udc.gal}
\address{EGR-XVR: Faculty of Mathematics,
University of Santiago de Compostela,
15782 Santiago de Compostela, Spain}
\email{eduardo.garcia.rio@usc.es, javier.valle@usc.es}
\address{PBG: Mathematics Department, University of Oregon, Eugene OR 97403-1222, USA}
\email{gilkey@uoregon.edu}
\thanks{Supported by projects MTM2016-75897-P and EM2014/009 (AEI/FEDER, UE)}
\subjclass[2010]{53C21, 53B30, 53C24, 53C44}
\keywords{Strong projective equivalence, Liouville's equivalence, projectively flat, affine quasi--Einstein equation.}

\begin{abstract}
We study the affine quasi-Einstein equation, a second order linear homogeneous equation, which is invariantly defined on any affine manifold. We prove that the space of solutions is finite-dimensional, and its dimension is a strongly projective invariant. Moreover the maximal dimension is shown to be achieved if and only if the manifold is strongly projectively flat.
\end{abstract}

\maketitle

\section{Introduction}

An affine manifold is a pair $\mathcal{M}=(M,\nabla)$ where $M$ is a smooth manifold of dimension $m$ and $\nabla$ is
a torsion free connection on the tangent bundle of $M$. Adopt the {\it Einstein convention} and sum over repeated
indices. Expand $\nabla_{\partial_{x^i}}\partial_{x^j}=\Gamma_{ij}{}^k\partial_{x^k}$ in a system of local coordinates
$\vec x=(x^1,\dots,x^m)$
to define the {\it Christoffel symbols} of the connection $\Gamma=(\Gamma_{ij}{}^k)$; the condition that $\nabla$ is torsion free
is then equivalent to the symmetry $\Gamma_{ij}{}^k=\Gamma_{ji}{}^k$. If $f\in C^\infty(M)$, 
then the {\it Hessian} $\mathcal{H}_\nabla f$ is the symmetric $(0,2)$-tensor
defined by setting
$$
\mathcal{H}_\nabla f:=\nabla^2f=(\partial_{x^i}\partial_{x^j}f-\Gamma_{ij}{}^k\partial_{x^k}f)\,dx^i\otimes dx^j\,.
$$
 Let $\rho_\nabla(x,y):=\operatorname{Tr}\{z\rightarrow R_\nabla(z,x)y\}$ be the {\it Ricci tensor}.
Since in general this need not be a symmetric 2-tensor, we introduce the symmetric and anti-symmetric Ricci tensors:
\begin{eqnarray*}
&&\textstyle\rho_{s,\nabla}(x,y):=\frac12\{\rho_\nabla(x,y)+\rho_\nabla(y,x)\},\\
&&\textstyle\rho_{a,\nabla}(x,y):=\frac12\{\rho_\nabla(x,y)-\rho_\nabla(y,x)\}\,.
\end{eqnarray*}
\subsection{The affine quasi-Einstein equation}
The {\it affine quasi-Einstein operator}
$ \mathfrak{Q}_{\mu,\nabla}$ is the linear second order partial differential operator
\begin{equation}\label{E1.a}
 \mathfrak{Q}_{\mu,\nabla} f:=\mathcal{H}_\nabla f-\mu f\rho_{s,\nabla}
 \text{ mapping }C^\infty(M)\text{ to }C^\infty(S^2M)
\end{equation}
where the eigenvalue $\mu$ is a parameter of the theory. For fixed $\mu$, this operator is natural in the category of affine manifolds
and this family of operators parametrizes, modulo scaling, all the natural second order differential operators from $C^\infty(M)$
to $C^\infty(S^2M)$.
We study the {\it affine quasi-Einstein equation} $ \mathfrak{Q}_{\mu,\nabla} f=0$, i.e.
\begin{equation}\label{E1.b}
\mathcal{H}_\nabla f=\mu f\rho_{s,\nabla}\,.
\end{equation}
We denote the space of all solutions to Equation~(\ref{E1.b}) by
$$
E(\mu,\nabla):=\ker( \mathfrak{Q}_{\mu,\nabla})=\{f\in \mathcal{C}^2(M):\mathcal{H}_\nabla f=\mu f\rho_{s,\nabla}\}\,.
$$
Similarly, if $P\in M$, let $E(P,\mu,\nabla)$ be the linear space of all germs of smooth functions based at $P$ satisfying
Equation~(\ref{E1.b}). Note that if $\rho_{s,\nabla}=0$, then $E(\mu,\nabla)=E(0,\nabla)$ for any $\mu$.
Also observe that $E(0,\nabla)$ is the space of \emph{affine Yamabe solitons} \cite{BGG16}.

The operator $\mathfrak{Q}_{\mu,\nabla}$ of Equation~(\ref{E1.a}) and the associated affine quasi-Einstein Equation~\eqref{E1.b}
are of sufficient interest in their own right in affine geometry to justify a foundational
paper of this nature. However Equation~\eqref{E1.b} also appears in the study of the
pseudo-Riemannian quasi--Einstein equation using the Riemannian extension; we postpone
until the end of the introduction a further discussion of this context to avoid interrupting the flow of our present discussion and
to establish the necessary notational conventions.

\subsection{Foundational results} We will establish the following result in Section~\ref{S2}:

\begin{theorem}\label{T1.1}
Let $\mathcal{M}=(M,\nabla)$ be an affine manifold. Let $f\in E(P,\mu,\nabla)$.
\begin{enumerate}
\item One has $f\in C^\infty(M)$. If $\mathcal{M}$ is real analytic, then $f$ is real analytic.
\item If $X$ is the germ of an affine Killing vector field based at $P$, then\newline$Xf\in E(P,\mu,\nabla)$.
\item If $f(P)=0$, and if $df(P)=0$, then $f\equiv 0$ near $P$.
\item One has $\dim\{E(P,\mu,\nabla)\}\le m+1$.
\item If $\mathcal{M}$ is simply connected and if $\dim\{E(P,\mu,\nabla)\}$ is constant on $M$, then $f$ extends
uniquely to an element of $E(\mu,\nabla)$.\end{enumerate}
\end{theorem}

\subsection{Projective equivalence}

\begin{definition}\rm We say that  $\nabla$ and $\tilde\nabla$  are {\it projectively equivalent} if there exists a $1$-form
$\omega$ so that
$\nabla_XY= \tilde\nabla_XY+\omega(X)Y+\omega(Y)X$ for all $X$ and $Y$. In this setting, we say that
$\omega$ {\it provides a projective equivalence from $\nabla$ to $\tilde\nabla$}; $-\omega$ then provides
a projective equivalence from $\tilde\nabla$ to $\nabla$. If $\omega$ is closed, we say
that $\nabla$ and $\tilde\nabla$ are {\it strongly projectively equivalent}. 
\end{definition}

If two projectively equivalent connections have symmetric Ricci tensors, then the 1-form $\omega$ giving the projective equivalence
is closed and the two connections are, in fact, strongly projectively equivalent (see \cite{Eisenhart, NS, S95} for more information).

\medskip

We say that $\nabla$ is {\it projectively flat} if $\nabla$
is projectively equivalent to a flat connection. 
Note that $\nabla$ is projectively
flat if and only if it is possible to choose a coordinate system so that the unparametrized geodesics of $\nabla$ are straight lines. 
Strongly projectively flat surfaces are characterized as follows (see \cite{Eisenhart, NS}).

\begin{lemma}\label{L4.1} Let $\mathcal{M}$ be an affine surface. 
\begin{enumerate}
\item Let $\omega$ provide a projective equivalence between $\nabla$ and a flat connection.
\begin{enumerate}
\item If $\rho_{a,\nabla}=0$, then $d\omega=0$ so $\nabla$ is strongly projectively flat.
\item If $d\omega=0$, then $\rho_\nabla$ and $\nabla\rho_\nabla$ are totally symmetric.
\end{enumerate}
\item If $\rho_\nabla$ and $\nabla\rho_\nabla$ are totally symmetric, then $\nabla$ is strongly projectively flat.
\end{enumerate}
\end{lemma}

Two projectively equivalent connections are said to be \emph{Liouville projectively equivalent} if their Ricci tensors coincide
(see, for example, \cite{EGR-KR, NS}). 
We will establish the following result in Section~\ref{S3.1}. It shows that $\dim\{E(-\frac1{m-1},\nabla)\}$ is a strong projective invariant
and that $\dim\{E(\mu,\nabla)\}$ for arbitrary $\mu$ is a strong Liouville projective invariant.

\begin{theorem}\label{T1.3}
Let $\mathcal{M}$ be an affine manifold of dimension $m$. Let $\mu_m=-\frac1{m-1}$. Let $\omega=dg$
provide a strong projective equivalence from $\nabla$ to $\tilde\nabla$.
\begin{enumerate}
\item The map $f\rightarrow e^gf$ is an isomorphism from $E(P,\mu_m,\nabla)$ to $E(P,\mu_m,\tilde\nabla)$.
\item The following assertions are equivalent:
\begin{enumerate}
\item $\rho_{s,\tilde\nabla}=\rho_{s,\nabla}$.
\item $\mathcal{H}_\nabla g-dg\otimes dg=0$.
\item $e^{ -g}\in E(P,0,\nabla)$.
\end{enumerate}
\item If any of the assertions in (2) hold, then the map $f\rightarrow e^gf$ is an isomorphism from $E(P,\mu,\nabla)$ to 
$E(P,\mu,\tilde\nabla)$
for any $\mu$.
\end{enumerate}\end{theorem}

\begin{remark}\label{R1.4}\rm If $\nabla$ and $\tilde\nabla$ are strongly projectively equivalent, then the alternating Ricci tensors coincide, i.e. 
$\rho_{a,\tilde\nabla}=\rho_{a,\nabla}$ (see \cite{S95}).
However, if $\nabla$ and $\tilde\nabla$ are only projectively equivalent, then the alternating Ricci tensors can differ and Theorem~\ref{T1.3}
can fail.
Let $\nabla$ be the
usual flat connection on $\mathbb{R}^2$ and let $\omega=x^2dx^1$ define a projective equivalence from $\nabla$ to $\tilde\nabla$.
It is a straightforward computation to see that $\dim\{E(P,-1,\nabla)\}=3$ and $\dim\{E(P,-1,\tilde\nabla)\}=0$.
Thus Theorem~\ref{T1.3} fails if we replace strong projective equivalence by projective equivalence.
Although the geodesic structure is unchanged, $\rho_{a,\tilde\nabla}\ne0$ in this instance and consequently the alternating Ricci tensor is not
preserved by projective equivalence either.
\end{remark}

We can say more about the geometry if $\dim\{E(P,\mu,\nabla)\}=m+1$ is extremal for some $\mu$. The eigenvalue $\mu_m:=-\frac1{m-1}$ plays a
distinguished role. We will establish the following result in Section~\ref{S3.2}:

\begin{theorem}\label{T1.5}
Let $\mathcal{M}$ be an affine manifold of dimension $m$. Let $\mu_m:=-\frac1{m-1}$.
\begin{enumerate}
\item $\mathcal{M}$ is strongly projectively flat if and only if $\dim\{E(\mu_m,\nabla)\}=m+1$.
\item If $\dim\{E(\mu,\nabla)\}=m+1$ for any $\mu$, then $\mathcal{M}$ is strongly projectively flat.
\item If $\dim\{E(\mu,\nabla)\}=m+1$ for $\mu\ne\mu_m$, then $\mathcal{M}$ is Ricci flat.
\item Suppose $\dim\{E(P,\mu_m,\nabla)\}=m+1$. One may choose a basis $\{\phi_0,\dots,\phi_m\}$ for $E(P,\mu_m,\nabla)$ so that
$\phi_0(P)\ne0$ and so that $\phi_i(P)=0$ for $i>0$. Set $z^i:=\phi_i/\phi_0$. Then $\vec z=(z^1,\dots,z^m)$
is a system of coordinates defined near $P$ such that the unparametrized geodesics of $\mathcal{M}$ are straight lines.\end{enumerate}\end{theorem}

\begin{remark}\rm The coordinates of Assertion~(4) are very much in the spirit of the Weierstrass preparation theorem for minimal surfaces;
geometrically meaningful local coordinates arise from the underlying analysis. 
\end{remark}

We will prove the following result in Section~\ref{S3.3}:

\begin{theorem}\label{T1.7}
Let $\mathcal{M}$ be an affine manifold of dimension $m$. Let $\mu_m:=-\frac1{m-1}$.
\begin{enumerate}
\item If $\mathcal{M}$ is strongly projectively equivalent to a connection $\tilde\nabla$ with $\rho_{s,\tilde\nabla}=0$, then $E(\mu_m,\nabla)\ne0$.
\item If there exists $f\in E(P,\mu_m,\nabla)$ with $f(P)\ne0$, then $\mathcal{M}$ is strongly projectively equivalent to a connection $\tilde\nabla$ with
$\rho_{s,\tilde\nabla}=0$ near $P$.
\end{enumerate}
\end{theorem}

Surface geometry is particularly tractable since the geometry is carried by the Ricci tensor. In this setting, $\mu_2=-1$ and one has
\begin{theorem}\label{T1.8}
Let $\mathcal{M}$ be an affine surface. Then $\dim\{E(-1,\nabla)\}\ne2$.
\end{theorem}

In the appendix, we will discuss some results concerning surface geometry in more detail.
We will use Theorem~\ref{TC.2} to show there are affine connections $\nabla_i$ on $\mathbb{R}^+\times\mathbb{R}$ such that
$$
\dim\{E(-1,\nabla_1)\}=0,\quad
\dim\{E(-1,\nabla_2)\}=1,\quad\dim\{E(-1,\nabla_3)\}=3\,.
$$
Thus the remaining values can all be attained. In Example~\ref{EB.2},
we will discuss a family of $3$-dimensional affine manifolds where $\dim\{E(-\frac12,\nabla)\}$ can be $0$, $1$, $2$, and $4$
but is never $3$. This suggests that for general $m$ one could show that $\dim\{E(P,-\frac1{m-1},\nabla)\}\ne m$ so this value is forbidden.
Our research continues on this problem.

\subsection{Riemannian extensions} 
The Riemannian extension provides a procedure to transfer information from affine geometry into neutral signature geometry in a natural way.
Let $\mathcal{M}=(M,\nabla)$ be an affine manifold. If $(x^1,\dots,x^m)$ are local
coordinates on $M$, 
let $(y_1,\dots,y_m)$ be the corresponding dual coordinates on the cotangent bundle $T^*M$; if $\omega$ is a 1-form, we can express
$\omega=y_idx^i$. Let $\Phi$ be an 
auxiliary symmetric $(0,2)$-tensor field in $M$. Let $\Gamma_{ij}{}^k$ be the Christoffel symbols of the connection $\nabla$.
The {\it deformed Riemannian extension} is the neutral signature metric on $T^*M$ which  is defined by setting \cite{afifi, PW}:
\begin{equation}\label{E1.c}
g_{\nabla,\Phi}= dx^i \otimes dy_i+dy^i \otimes dx_i   +\left\{\Phi_{ij}-2y_k\, \Gamma_{ij}{}^k\right\}dx^i \otimes dx^j\,.
\end{equation}
This is invariantly defined \cite{walker-metrics}. Let $\pi$ be the canonical projection
from $T^*M$ to $M$. If $f\in C^\infty(M)$, then
$$
\mathcal{H}_{g_{\nabla,\Phi}}\pi^*f=\pi^*\mathcal{H}_\nabla f,\quad\rho_{g_{\nabla,\Phi}}=2\pi^*\rho_{s,\nabla},\quad\|d\pi^*f\|_{g_{\nabla,\Phi}}^2=0\,.
$$

Let $\mathcal{N}:=(N,g_N,\Psi,\mu)$ where $(N,g_N)$ is a pseudo-Riemannian manifold of dimension $n$, where
$\Psi\in C^\infty(N)$, and where
$\mu\in \mathbb{R}$. We say that $\mathcal{N}$ is a {\it quasi-Einstein manifold} if
\begin{equation}\label{E1.d}
\mathcal{H}_{g_N}\Psi+\rho_{g_N}-\mu\, d\Psi\otimes d\Psi=\lambda\, g_{N}\text{ for some }\lambda\in\mathbb{R}\,.
\end{equation}
One has the following link between deformed Riemannian extensions and quasi-Einstein structures \cite{BGGV17}:

\begin{theorem}
\ \begin{enumerate}
\item Let $\mathcal{M}=(M,\nabla)$ be an affine surface, let $\psi\in C^\infty(M)$, and let 
$\mu\in\mathbb{R}$. If $\mathcal{H}_\nabla \psi+2\rho_{s,\nabla}-\mu\, d\psi\otimes d\psi=0$,
then $(T^*M,g_{\nabla,\Phi},\pi^*\psi,\mu)$ is a
self-dual quasi-Einstein manifold  with $\| d\pi^*\psi\|_{g_{\nabla,\Phi}}^2=0$ and $\lambda=0$, for any~$\Phi$.
\item Let $(N,g_N,\Psi,\mu)$ be a self-dual quasi-Einstein  manifold of signature $(2,2)$ with $\mu\neq -\frac{1}{2}$ and $\| d\Psi\|^2_{g_N}=0$
which is not Ricci flat. Then $\lambda=0$ and $(N,g_N,\Psi,\mu)$ is locally isometric to a manifold which has the form given in
Assertion~(1).
\end{enumerate}\end{theorem}

We suppose $\mu\ne0$ and make the change of variables $f=e^{-\frac12\mu\,\psi}$. The equation 
$\mathcal{H}_\nabla \psi+2\rho_{s,\nabla}-\mu\, d\psi\otimes d\psi=0$ then becomes 
$\mathcal{H}_\nabla f=\mu\, f\,\rho_{s,\nabla}$.
This is the {\it affine quasi-Einstein equation} given in \eqref{E1.b}. 
Let $f>0$ be a smooth function on $M$. Express $\pi^*f=e^{-\mu\,F}$ for some $F\in C^\infty(T^*M)$ and $\mu\neq 0$. Then $F$ solves Equation~(\ref{E1.d}) in $(T^*M,g_{\nabla,\Phi})$ if and only if $f\in E(2\mu,\nabla)$.

\begin{remark}\rm
The eigenvalue $\mu_m=-\frac{1}{m-1}$, which plays a role in the projective structure of $(M,\nabla)$,  is linked to some geometric properties of the deformed Riemannian extensions $\mathcal{N}:=(T^*M,g_{\nabla,\Phi})$: if $\operatorname{dim}\{ E(\mu_m,\nabla)\}\geq 1$, then $\mathcal{N}$ is conformally Einstein \cite{BGGV17}.

Afifi \cite{afifi} showed that if a deformed Riemannian extension $g_{\nabla,\Phi}$ given by Equation~\eqref{E1.c} is locally conformally flat, then $\nabla$ is projectively flat with symmetric Ricci tensor. 
Theorem \ref{TC.2} shows the existence of surfaces with $\operatorname{dim}\{E(-1,\nabla)\}=1$. The corresponding deformed Riemannian extensions $(T^*M,g_{\nabla,\Phi})$ are conformally Einstein but not conformally flat for any deformation tensor $\Phi$ by Theorem \ref{T1.5}.
\end{remark}

\begin{remark}\rm
Two metrics in the same conformal class are said to be \emph{Liouville equivalent} if their Ricci tensors coincide (see \cite{EGR-KR, EGR-KR3}).
Let $\nabla$ and $\nabla^{dg}$ be strongly projectively equivalent connections. Then the corresponding Riemannian extensions $g_\nabla$ and $g_{\nabla^{dg}}$ are conformally equivalent
(just considering the transformation $(x^k,y_k)\mapsto (x^k,e^{2g}y_k)$). Moreover, $g_\nabla$ and $g_{\nabla^{dg}}$ are Liouville equivalent if and only if  $\mathcal{H}_\nabla g=dg\otimes dg$.
Therefore, from the pseudo-Riemannian point of view, affine Yamabe solitons  $\phi\in E(0,\nabla)$ determine Liouville transformations  of the Riemannian extension $g_\nabla$ (see Assertion (2) in Theorem \ref{T1.3}).

The projective deformations in Example \ref{E-A.2} and Example \ref{E-A.3} induce Liouville equivalent Riemannian extensions \cite{EGR-KR}. Therefore it follows from \cite[Corollary 2]{EGR-KR3} that none of the Riemannian extensions $g_\nabla$ and $g_{\nabla^{\omega}}$ is geodesically complete.
\end{remark}

\begin{remark}\rm
There is a close connection between quasi-Einstein structures and warped product Einstein metrics (see \cite{BGGV17} and references therein).
The warping function of any Einstein warped product is a solution of Equation~\eqref{E1.d} with $\mu=\frac{1}{k}$, $k\in\mathbb{N}$.
Conversely, if $f\in E(\frac1{2k},\nabla)$ for some positive integer $k$ and if $\mathcal{E}$ is a Ricci flat manifold of dimension $k$, then the warped product 
$\mathcal{N}\times_{\pi^*f}\mathcal{E}$ with base manifold $\mathcal{N}:=(T^*M,g_{\nabla,\Phi})$ is Ricci flat.
Theorem \ref{TA.1}-(3) and Theorem \ref{TC.3} show that there exist homogeneous surfaces with
$\operatorname{dim}\{E(\mu,\nabla)\}\geq 1$ for arbitrary $\mu=\frac{1}{2k}$.
\end{remark}
 
\section{The proof of Theorem~\ref{T1.1}}\label{S2}
We establish the assertions of Theorem~\ref{T1.1} seriatim. 
\subsection{Smoothness properties of solutions to Equation~(\ref{E1.b})} Introduce local coordinates $x=(x^1,\dots,x^m)$. Let 
$\mathfrak{Q}_{\mu,\nabla,ij}$ be the components of the quasi-Einstein operator of Equation~(\ref{E1.a}).
Let
$$
D_\mu:=\operatorname{Tr}\{ \mathfrak{Q}_{\mu,\nabla}\}=\sum_{i=1}^m\mathfrak{Q}_{\mu,\nabla,ii}
=\sum_{i=1}^m\partial_{x^ix^i}+\sum_{i=1}^m\sum_{j=1}^m\Gamma_{ii}{}^j\partial_{x^j}-\mu\sum_{i=1}^m\rho_{ii}\,.
$$
The operator $D_\mu$ is then an elliptic second order partial differential operator. Let $f\in C^{ 2}(M)$ satisfy $ \mathfrak{Q}_{\mu,\nabla} f=0$. One then has
$D_\mu f=0$ and standard elliptic theory shows $f\in C^\infty(M)$. Suppose in addition that 
the underlying structure is real analytic. It then follows that
$D_\mu$ is analytic-hypoelliptic and hence $D_\mu f=0$ implies $f$ is real analytic, see, for example, the discussion in \cite{C92,T71}.

\subsection{Affine Killing vector fields} Let $\Phi_t^X$ be the 1-parameter flow associated with an affine Killing vector field $X$.
Then $\Phi_t^X$ commutes with $\nabla$ and hence with $ \mathfrak{Q}_{\mu,\nabla}$ for all $t$. Thus if $f\in E(\mu,\nabla)$, then
$(\Phi_t^X)^*f\in E(\mu,\nabla)$ for any $t$. Differentiating this relation with respect to $t$ and setting $t=0$ then shows
$Xf\in E(\mu,\nabla)$ as desired.

\subsection{Initial conditions}
We wish to show that if $f\in E(P,\mu,\nabla)$, if $f(P)=0$, and if $df(P)=0$, then $f\equiv0$.
In the real analytic category, this is immediate as we can use Equation~(\ref{E1.b}) to show all the higher derivatives vanish.
Our task is to give a different derivation in the $C^\infty$ context. To simplify the discussion, we shall assume $m=2$. Introduce
local coordinates $(x^1,x^2)$ on $M$ centered at $P$. Let $B_\varepsilon(0)$ be the ball of radius $\varepsilon$ about the origin.
Assume that $f\in E(P,\mu,\nabla)$ satisfies $f(0)=df(0)=0$. We will show there exists $\varepsilon>0$ so that $f\equiv0$ on
$B_\varepsilon(0)$. Choose $T$ and $\varepsilon$ so that
\begin{equation}\label{E2.a}
\textstyle\frac13<T,\quad|\Gamma_{ij}{}^k(x)|\le T,\quad|\mu\,\rho_{s, \nabla,ij}(x)|\le T\text{ for all }x\in B_\varepsilon(0),\quad\varepsilon<\frac1{12T}\,.
\end{equation}
Let
$$
\|f\|_1:=\sup_{x\in B_\varepsilon(0)}\left\{|\partial_{x^1}f(x)|,|\partial_{x^2}f(x)|,|f(x)|\right\}\,.
$$
Let $\vec x=(a,b)\in B_\varepsilon(0)$. Let $\gamma(t)=t\vec x$. We use Equation~(\ref{E2.a}) to estimate:
\begin{eqnarray*}
\left|\partial_t\partial_{x^1}f\right|(t\vec x) &=&\left|a\partial_{x^1x^1}f+b\partial_{x^1x^2}f\right|(t\vec x)\le\left|a\partial_{x^1x^1}f\right|(t\vec x)
+\left|b\partial_{x^1x^2}f\right|(t\vec x)\\
&=&\phantom{+}|a|\cdot\left|\Gamma_{11}{}^1\partial_{x^1}f+\Gamma_{11}{}^2\partial_{x^2}f+f\mu\rho_{s\nabla,11}\right|(t\vec x)\\
&&+|b|\cdot\left|\Gamma_{12}{}^1\partial_{x^1}f+\Gamma_{12}{}^2\partial_{x^2}f+f\mu\rho_{s\nabla,12}\right|(t\vec x)\\
&\le&3(|a|+|b|)T\|f\|_1\le6\,\varepsilon\, T\|f\|_1\,.
\end{eqnarray*}
As $\partial_{x^1}f(0)=0$, we may use the Fundamental Theorem of Calculus to estimate:
$$
|\partial_{x^1}f(\vec x)|\le\int_{t=0}^1\left|\partial_t\partial_{x^1}f(t\vec x)\right|dt\le\int_{t=0}^16\,\varepsilon\, T\,\|f\|_1dt=6\,\varepsilon\, T\|f\|_1\,.
$$
We show similarly that $|\partial_{x^2}f(\vec x)|\le6\,\varepsilon\, T\|f_1\|$. 
Finally, since $\frac{1}{3}<T$ and since $f(0)=0$, we estimate
$$
\left|f(\vec x)\right|\le\int_{t=0}^1\left|\partial_tf(t\vec x)\right|dt\le\int_{t=0}^1(|a|+|b|)\|f\|_1dt\le2\,\varepsilon\,\|f\|_1\le 6\,\varepsilon\, T\|f\|_1\,.
$$
Consequently, $\|f\|_1\le 6\,\varepsilon\, T\|f\|_1$. Since $6\,\varepsilon\, T<\frac12$, $\|f\|_1\le\frac12\|f\|_1$. This
implies $\|f\|_1=0$ on $B_\varepsilon(0)$ and proves Theorem~\ref{T1.1}~(3).

\subsection{Estimating the dimension of $E(P,\mu,\nabla)$}
Let $f\in E(P,\mu,\nabla)$. By Theorem~\ref{T1.1}~(3),  $f$ is determined by $f(P)$ and $df(P)$. Assertion~(4) now follows.

\subsection{Extending solutions to Equation~(\ref{E1.b})} The
final assertion of Theorem~\ref{T1.1} follows using exactly the same arguments of ``analytic continuation" that were used to prove similar assertions
for Killing vector fields or affine Killing vector fields (see~\cite{N60}).
\qed

\section{Projective equivalence}\label{S3}
In what follows, it will be convenient to work with just one component. Suppose that $\Phi$ is a symmetric $(0,2)$-tensor
defined on some vector space $V$ and suppose that one could show that $\Phi_{11}=0$ relative to any basis. It then follows that $\Phi=0$;
this process is called {\it polarization}.
If $\mathcal{M}=(M,\nabla)$ is an affine manifold, then
\begin{eqnarray*}
&&R_{\nabla,ijk}{}^l=
\partial_{x^i}\Gamma_{jk}{}^l-\partial_{x^j}\Gamma_{ik}{}^l+\Gamma_{in}{}^l\Gamma_{jk}{}^n-\Gamma_{jn}{}^l\Gamma_{ik}{}^n,\\
&&\rho_{\nabla,jk}=
\partial_{x^i}\Gamma_{jk}{}^i-\partial_{x^j}\Gamma_{ik}{}^i+\Gamma_{in}{}^i\Gamma_{jk}{}^n-\Gamma_{jn}{}^i\Gamma_{ik}{}^n\,.
\end{eqnarray*}

\begin{lemma}\label{L3.1}
Let $\omega=dg$ provide a strong projective equivalence from $\nabla$ to $\tilde\nabla$.
\begin{enumerate}
\item $\rho_{s,\tilde\nabla}=\rho_{s,\nabla}-(m-1)\{\mathcal{H}_\nabla g-dg\otimes dg$\}. 
\item If $\mu=-\frac1{m-1}$ or if $\mathcal{H}_\nabla g-dg\otimes dg=0$, then $ \mathfrak{Q}_{\mu,\nabla} =e^{-g}\, {\mathfrak{Q}}_{\mu,\tilde{\nabla}}\, e^g$.
\end{enumerate}
\end{lemma}

\begin{proof}
Assume $\tilde\nabla_XY=\nabla_XY+dg(X)Y+dg(Y)X$, i.e.
$$
\tilde\Gamma_{ij}{}^k=\Gamma_{ij}{}^k+\delta_i^k\,\partial_{x^j}g+\delta_j^k\,\partial_{x^i}g\,.
$$
Fix a point $P$ of $M$. Since we are working in the category
of connections without torsion, we can choose a coordinate system so $\Gamma(P)=0$. We compute at the point $P$ and set 
$\Gamma_{ij}{}^k(P)=0$ to see
$$
\begin{array}{rcl}
\rho_{\tilde{\nabla,}11}(P) &=&\{\partial_{x^i}\tilde\Gamma_{11}{}^i-\partial_{x^1}\tilde\Gamma_{i1}{}^i +\tilde\Gamma_{in}{}^i\tilde\Gamma_{11}{}^n
-\tilde\Gamma_{1n}{}^i\tilde\Gamma_{i1}{}^n\}(P)\\
\noalign{\medskip}
\quad 
&=&\{\partial_{x^i}\Gamma_{11}{}^i-\partial_{x^1}\Gamma_{i1}{}^i +(1-m)\partial_{x^1x^1}g\\
\noalign{\medskip}
\quad 
&&
\phantom{\{\partial_{x^i}\Gamma_{11}{}^i}
+2(m+1)(\partial_{x^1}g)^2-(m+3)(\partial_{x^1}g)^2\}(P)\\
\noalign{\medskip}
\quad 
&=&\{\rho_{\nabla, 11}+(m-1)((\partial_{x^1}g)^2-\partial_{x^1x^1}g)\}(P)\\
\noalign{\medskip}
\quad
&=&\{\rho_\nabla-(m-1)(\mathcal{H}_\nabla g-dg\otimes dg)\}_{11}(P)\,.
\end{array}
$$
Polarizing this identity establishes Assertion~(1).  To prove Assertion~(2),  
we examine
$\mathfrak{Q}_{\mu,\nabla,11}$ and $\{e^{-g}\,{\mathfrak{Q}_{\mu,\tilde\nabla}}\,e^g\}_{11}$ at $P$. 
We compute:
$$
\begin{array}{l}
\{e^{-g}\mathcal{H}_{\tilde\nabla,11}e^gf\}(P)
=\{e^{-g}\partial_{x^1x^1}(fe^g)-\tilde\Gamma_{11}{}^ke^{-g}\partial_{x^k}(fe^g)\}(P)
\\[0.05in]
\quad
=\{\partial_{x^1x^1}f+2\partial_{x^1}f\partial_{x^1}g +f\partial_{x^1x^1}g+f(\partial_{x^1}g)^2 -2\partial_{x^1}g(\partial_{x^1}f+f\partial_{x^1}g)\}(P)
\\[0.05in]\quad
=\{\mathcal{H}_{\nabla,11}f+f(\partial_{x^1x^1}g-(\partial_{x^1}g)^2)\}(P).
\end{array}
$$
We complete the proof by polarizing the resulting identity:
\medbreak\qquad
$\{e^{-g}\mathfrak{Q}_{\mu,\tilde\nabla}\, e^g f- \mathfrak{Q}_{\mu,\nabla} f\}_{11}(P)$
\medbreak\qquad\quad
$=\{e^{-g}(\mathcal{H}_{\tilde\nabla}e^gf-\mu\rho_{s,\tilde{\nabla}}e^gf)_{11}
-(\mathcal{H}_{\nabla}f-\mu\rho_{s,\nabla}f)_{11}\}(P)$
\medbreak\qquad\quad
$=\{f(1+(m-1)\mu)(\partial_{x^1x^1}g-(\partial_{x^1}g)^2)\}(P)$.
\end{proof}
\subsection{Proof of Theorem~\ref{T1.3}}\label{S3.1}
Theorem~\ref{T1.3}~(1) is immediate from the intertwining relation of
Lemma~\ref{L3.1}~(2). The equivalence of Assertion~(2a) and Assertion~(2b) follows from Lemma~\ref{L3.1}~(1). The equivalence
of Assertion~(2b) and Assertion~(2c) follows by noting
$$
\mathfrak{Q}_{0,\nabla}(e^{-g})=\mathcal{H}_\nabla (e^{-g})=-e^{-g}\{\mathcal{H}_\nabla (g)-dg\otimes dg\}\,.
$$
Assertion~(3) now follows from Assertion~(2b) and Lemma~\ref{L3.1}~(2).\qed

\subsection{Proof of Theorem~\ref{T1.5}}\label{S3.2} Let $\mu_m:=-\frac1{m-1}$.
To prove Assertion~(1), we suppose that $\mathcal{M}$ is strongly projectively flat, i.e. $\nabla$ is strongly projectively equivalent
to a flat connection $\tilde\nabla$. Under this assumption, there are local coordinates around $P\in M$ so that the Christoffel symbols
$\tilde{\Gamma}_{ij}{}^k$ vanish identically. Thus,
$$
E(P,\mu_m,\tilde\nabla)=\operatorname{Span}\{1,x^1,\dots,x^m\}\,.
$$
Consequently, by Theorem~\ref{T1.3}, $\dim\{E(\mu_m,\nabla)\}=\dim\{E(\mu_m,\tilde\nabla)\}=m+1$. 
Next, assume that $\dim\{E(P,\mu,\nabla)\}=m+1$ for some $\mu$.
If $\phi\in E(P,\mu,\nabla)$, let
$$
\Theta(\phi):=(\phi,\partial_{x^1}\phi,\dots,\partial_{x^m}\phi)(P)\in\mathbb{R}^{m+1}\,.
$$
This vanishes if and only if $\phi\equiv0$. For dimensional reasons, $\Theta$ must be an isomorphism. 
Let $e_i$ be the standard basis for $\mathbb{R}^{m+1}$ and let $\phi_i=\Theta^{-1}(e_i)$
be the corresponding basis for $E(P,\mu,\nabla)$. Since $\Theta(\phi_i)=e_i$, we have
$$\begin{array}{ccccc}
\phi_0(P)=1,&\partial_{x^1}\phi_0(P)=0,&\partial_{x^2}\phi_0(P)=0,&\dots,&\partial_{x^m}\phi_0(P)=0,\\
\phi_1(P)=0,&\partial_{x^1}\phi_1(P)=1,&\partial_{x^2}\phi_1(P)=0,&\dots,&\partial_{x^m}\phi_1(P)=0,\\
\phi_2(P)=0,&\partial_{x^1}\phi_{2}(P)=0,&\partial_{x^2}\phi_{2}(P)=1,&\dots,&\partial_{x^m}\phi_{2}(P)=0,\\
\dotfill&\dotfill&\dotfill&\dotfill&\dotfill,\\
\phi_m(P)=0,&\partial_{x^1}\phi_{m}(P)=0,&\partial_{x^2}\phi_{m}(P)=0,&\dots,&\partial_{x^m}\phi_{m}(P)=1.
\end{array}$$
Set
$z^1:=\phi_1/\phi_0$, \dots, $z^m:=\phi_m/\phi_0$. We then have
$$\begin{array}{ccccc}
z^1(P)=0,&\partial_{x^1}z^1(P)=1,&\partial_{x^2}z^1(P)=0,&\dots,&\partial_{x^m}z^1(P)=0,\\
z^2(P)=0,&\partial_{x^1}z^{2}(P)=0,&\partial_{x^2}z^{2}(P)=1,&\dots,&\partial_{x^m}z^{2}(P)=0,\\
\dotfill&\dotfill&\dotfill&\dotfill&\dotfill\\
z^m(P)=0,&\partial_{x^1}z^{ m}(P)=0,&\partial_{x^2}z^{ m}(P)=0,&\dots,&\partial_{x^m}z^{ m}(P)=1.
\end{array}$$
Thus $\vec z(P)=0$ and $d\vec z(P)=\operatorname{id}$. Hence this is an admissible change of coordinates 
centered at $P$. Set $g=\log(\phi_0)$.
We then obtain 
\begin{equation}\label{E3.a}
E(P,\mu,\nabla)=e^g\operatorname{Span}\{1,z^1,\dots,z^m\}\,.
\end{equation}
We have that
$$
\mathfrak{Q}_{\mu,\nabla}(e^g)=\mathfrak{Q}_{\mu,\nabla}(\phi_0)=0\text{ and }\mathfrak{Q}_{\mu,\nabla}(z^ke^g)=\mathfrak{Q}_{\mu,\nabla}(\phi_k)=0\,.
$$
We set $e^{-g}\{z^k\mathfrak{Q}_{\mu,\nabla}(e^g)- \mathfrak{Q}_{\mu,\nabla}(z^ke^g)\}=0$ and examine
the resulting relations. Fix $i$, $j$, and $k$. We compute:
\medbreak\quad $e^{-g}\left\{\partial_{ z^i}\partial_{z^j}(z^ke^g)- z^k\partial_{ z^i}\partial_{ z^j}(e^g)\right\}=\delta_j^k\,{ \partial_{z^i}}g+\delta_i^k\,{\partial_{z^j}}g$,
\medbreak\quad $-e^{-g}\Gamma_{ij}{}^\ell\left\{\partial_{ z^\ell}( z^ke^g)- z^k\partial_{ z^\ell}(e^g)\right\}=-\Gamma_{ij}{}^k$,
\medbreak\quad 
$e^{-g}\left\{\mu\rho_{ s,\nabla}( z^ke^g)- z^k\mu\rho_{s,\nabla} e^g\right\}=0$,
\medbreak\quad $0=e^g\left\{\mathfrak{Q}_{\mu,\nabla}( z^ke^g)- z^k\mathfrak{Q}_{\mu,\nabla}(e^g)\right\}_{ij}=\delta_j^k\,{ \partial_{z^i}}g+\delta_i^k\,{\partial_{z^j}}g-\Gamma_{ij}{}^k$.
\medbreak\noindent 
Let $\tilde\Gamma_{ij}{}^k=0$ define a flat connection $\tilde\nabla$. We have 
$\Gamma_{ij}{}^k=\tilde\Gamma_{ij}{}^k+\delta_j^k\,{ \partial_{z^i}}g+\delta_i^k\,{\partial_{z^j}}g$ so $dg$ provides a strong projective equivalence from
$\tilde\nabla$ to $\nabla$. Consequently, $\nabla$ is strongly projectively flat. This establishes Assertions (1) and (2). 

Furthermore, by Theorem~\ref{T1.3}, $\tilde f\rightarrow e^g\tilde f$ is an isomorphism from $E(P,\mu_m,\tilde\nabla)$ to
$E(P,\mu_m,\nabla)$. Since $1\in E(P,\mu_m,\tilde\nabla)$,  $e^g\in E(P,\mu_m,\nabla)$. By Equation~(\ref{E3.a}),
$e^g\in E(P,\mu,\nabla)$. This means
$\mathcal{H}_\nabla e^g=e^g\mu_m\rho_{s,\nabla}$ and $\mathcal{H}_\nabla e^g=e^g\mu\rho_{s,\nabla}$. Since $\mu\ne\mu_m$, 
this implies $\rho_{s,\nabla}=0$. Since $\tilde\nabla$ is flat, $\rho_{a,\tilde\nabla}=0$. By Remark~\ref{R1.4}, 
the alternating Ricci tensor is preserved by strong projective equivalence. Consequently,
$\rho_{a,\nabla}=0$ as well. This implies $\nabla$ is Ricci flat which establishes Assertion~(3).
Assertion~(4) follows from the discussion given above.\qed

\subsection{The proof of Theorem~\ref{T1.7}}\label{S3.3} Suppose $dg$ provides a strong projective equivalence from
$\nabla$ to a connection $\tilde\nabla$ with $\rho_{s,\tilde\nabla}=0$. We use Lemma~\ref{L3.1} to see that 
$\mathcal{H}_\nabla g-dg\otimes dg=\frac1{m-1}\rho_{s,\nabla}$. Set $f=e^{-g}$. Then 
$$
\textstyle\mathcal{H}_\nabla f=e^{-g}\{-\mathcal{H}_\nabla g+dg\otimes dg\}=-\frac1{m-1}f\rho_{s,\nabla}
$$
so $f\in E(\mu_m,\nabla)$ is non-trivial. This establishes Assertion~(1) of Theorem~\ref{T1.7}. Conversely, of course, if $f\in E(P,\mu_m,\nabla)$
satisfies $f(P)\ne0$, then we may assume $f(P)>0$ and set $g=-\log(f)$. Reversing the argument then establishes Assertion~(2) of
Theorem~\ref{T1.7}.\qed

\subsection{The proof of Theorem~\ref{T1.8}}

Let $m=2$ and $\mu_2=-1$. Suppose to the contrary that $\dim\{E(P,-1,\nabla)\}=2$; we argue for a contradiction. 
Suppose first that $f(P)>0$ for some $f\in E(P,-1,\nabla)$.
Express $f=e^g$ near $P$. Let $-dg$ provide a strong projective equivalence from $\nabla$ to $\tilde\nabla$.
By Theorem~\ref{T1.3}, $1=e^{-g}f\in E(P,-1,\tilde\nabla)$. It now follows that $\rho_{s,\tilde\nabla}=0$. By Remark~\ref{R1.4},
$\rho_{a,\nabla}=\rho_{a,\tilde\nabla}$. Thus if $\rho_{a,\nabla}=0$, then $\tilde\nabla$ is Ricci flat and hence, since $m=2$, $\tilde\nabla$ is flat. 
Consequently, 
we apply Theorem~\ref{T1.5} to conclude $\dim\{E(P,-1,\nabla)\}=3$ contrary to our assumption. 

We suppose therefore that $\rho_{a,\tilde\nabla}=\rho_{a,\nabla}$ is non-trivial and $\rho_{s,\tilde\nabla}=0$.
Hence $\rho_{a,\tilde\nabla}$ defines a nonzero two-form, which shows that the curvature tensor is recurrent. Thus $(M,\tilde\nabla)$ is locally described
by the work of Wong \cite[Theorem 4.2]{W64}. Recently, Derdzinski \cite[Theorem 6.1]{derdzinski} has shown that local coordinates can be specialized so that
the only non-zero Christoffel symbols are 
$\Gamma_{11}{}^1 =-\partial_{x^1}\phi$ and $\Gamma_{22}{}^2 =\partial_{x^2}\phi$. We have
$E(P,\mu_2,\tilde\nabla)=\ker(\mathcal{H}_{\tilde\nabla})$. Since this is, by assumption, 2-dimensional, we can apply Theorem~\ref{T1.1} to choose
$\tilde f\in\ker(\mathcal{H}_{\tilde\nabla})$ so that $d\tilde f(P)\ne0$.  We compute
$$
\begin{array}{l}
0=\mathcal{H}_{\tilde\nabla,11}\tilde f={\partial_{x^1x^1}}\tilde f+{\partial_{x^1}}\phi\, {\partial_{x^1}}\tilde f,\,\,\,\,
0=\mathcal{H}_{\tilde\nabla,22}\tilde f={\partial_{x^2x^2}}\tilde f{ -}{\partial_{x^2}}\phi\, {\partial_{x^2}}\tilde f,\\
\noalign{\medskip}
0=\mathcal{H}_{\tilde\nabla,12}\tilde f={\partial_{x^1x^2}}\tilde f\,.
\end{array}
$$
The relation ${\partial_{x^1x^2}}\tilde f=0$ implies $\tilde f(x^1,x^2)=a(x^1)+b(x^2)$.
Differentiating the remaining relations with respect to $x^2$ and $x^1$, respectively, yields
$$
{\partial_{x^1x^2}\phi} \, a^\prime(x^1)=0\text{ and }-{\partial_{x^1x^2}\phi} \, b^\prime(x^2)=0\,.
$$
By assumption, $d\tilde f(P)\ne0$ and thus $(a^\prime(0),b^\prime(0))\ne(0,0)$.
Thus ${\partial_{x^1x^2}\phi}$ vanishes identically at $P$. This implies the geometry is flat and $\rho_{a,\tilde\nabla}=0$ contrary to our assumption.

Suppose $f(P)=0$ for every $f\in E(P,-1,\nabla)$ and $\dim\{E(P,-1,\nabla)\}=2$. Let $\{f_1,f_2\}$ be a basis for $E(P,-1,\nabla)$.
Since $f_i(P)=0$, we may apply Theorem~\ref{T1.1} to see $df_1(P)$ and $df_2(P)$ are linearly independent. Thus we can choose local
coordinates centered at $P$ so that $E(P,-1,\nabla)=\operatorname{Span}\{x^1,x^2\}$. If $Q\ne P$, then $\dim\{E({ Q},-1,\nabla)\}\ge2$ and
there exists a non-vanishing element $f_Q$ of $E(Q,-1,\nabla)$ with $f_Q(Q)\ne0$. The argument given above shows that $\dim\{E(Q,-1,\nabla)\}=3$.
  Thus $\nabla$ is strongly projectively flat near $Q$
 so by Lemma~\ref{L4.1}~(1), $\rho_\nabla$ and $\nabla\rho_\nabla$ are totally symmetric at $Q$. Thus, by continuity, the same holds at $P$.
 Thus by Lemma~\ref{L4.1}~(2), we can conclude that $\nabla$ is strongly projectively flat on a neighborhood of $P$ and $\dim\{E(P,-1,\nabla)\}=3$ contrary to our assumption.

\appendix\section{Locally homogeneous affine surfaces}

We say that $\mathcal{M}=(M,\nabla)$ is {\it locally homogeneous} if, given any two points of $M$, there is the germ of a diffeomorphism
$T$ taking one point to another with $T^*\nabla=\nabla$. Locally homogeneous affine surfaces have
been classified by Opozda \cite{Op04}.
Let $\mathcal{M}$ be a locally homogeneous affine surface which is not flat, i.e. has non-vanishing
Ricci tensor. Then at least one of the following
three possibilities holds, which are not exclusive, and which describe the local geometry:
\smallbreak\noindent{\bf Type~$\mathcal{A}$}: There exist local coordinates $(x^1,x^2)$ so that
$\Gamma_{ij}{}^k$ are constant.
\smallbreak\noindent{\bf Type~$\mathcal{B}$:} There exist local coordinates $(x^1,x^2)$ so that
$\Gamma_{ij}{}^k=(x^1)^{-1}C_{ij}{}^k$ where $C_{ij}{}^k$ are constant.
\smallbreak\noindent{\bf Type~$\mathcal{C}$}: $\nabla$ is the Levi-Civita connection of a metric of constant sectional curvature.

\medskip
In Section~\ref{S-A} and Section~\ref{S-C}, we present 2 and 3-dimensional solutions to the affine quasi-Einstein equation \eqref{E1.b} which are Type~$\mathcal{A}$ and Type~$\mathcal{B}$ geometries. Our account here is purely expository
to illustrate some of the phenomena which occur; we shall postpone the proofs of these results for a subsequent paper \cite{BGGV17x}.
In each  case  we consider the essentially different  eigenvalues  $\mu=0$, $\mu_m=-\frac{1}{m-1}$, and $\mu\neq 0, -\frac{1}{m-1}$  separately. 

\subsection{Type $\mathcal{A}$ surfaces}\label{S-A}
Let $\mathcal{M}=(\mathbb{R}^2,\nabla)$ be a Type~$\mathcal{A}$ surface model which is not flat; the Christoffel symbols satisfy
$\Gamma_{ij}{}^k=\Gamma_{ji}{}^k\in\mathbb{R}$. 
Any Type $\mathcal{A}$ surface is projectively flat with symmetric Ricci tensor \cite{BGG16}, thus strongly projectively flat.

\begin{theorem}\label{TA.1}
Let $\mathcal{M}$ be a Type~$\mathcal{A}$ surface model. 
\begin{enumerate}
\item Let $\mu=0$. Then $E(0,\nabla)=\operatorname{Span}\{1\}$ or, up to linear equivalence, one of the following holds:
\begin{enumerate}
\item $\Gamma_{11}{}^1=1$, $\Gamma_{12}{}^1=0$, $\Gamma_{22}{}^1=0$, and $E(0,\nabla)=\operatorname{Span}\{1,e^{x^1}\}$.
\smallbreak\item $\Gamma_{11}{}^1=\Gamma_{12}{}^1=\Gamma_{22}{}^1=0$, and $E(0,\nabla)=\operatorname{Span}\{1,x^1\}$.
\end{enumerate}
\item Let $\mu=-1$. Then $\dim\{E(-1,\nabla)\}=3$.
\item Let $\mu\ne 0,-1$. Then $\dim\{E(\mu,\nabla)\}=\left\{\begin{array}{ll}
2\text{ if } \operatorname{Rank}\{\rho_{\nabla}\}=1\\
0\text{ if } \operatorname{Rank}\{\rho_{\nabla}\}=2
\end{array}\right\}$.
\end{enumerate}\end{theorem}
 
 We can use Theorem~\ref{T1.3} to construct non-trivial projective deformations.
 \begin{example}\rm\label{E-A.2}
 We set
$\Gamma_{11}{}^1=1$, $\Gamma_{12}{}^1=0$, $\Gamma_{22}{}^1=0$ as in Theorem \ref{TA.1}-(1.a).  Then $\phi(x^1,x^2)=a+e^{x^1}\in E(0,\nabla)$. Following Theorem \ref{T1.3},
set $g=-\log\phi(x^1,x^2)$ and consider the strongly projectively equivalent connection  $\tilde\nabla$ determined by the 1-form $\omega=dg$. We have $\rho_{\nabla}=\rho_{\tilde{\nabla}}=\rho_{\nabla,11}dx^1\otimes dx^1$;
both $\nabla\rho_{\nabla}$ and $\tilde\nabla\rho_{\tilde\nabla}$ are multiples of $dx^1\otimes dx^1\otimes dx^1$. 
Thus  $\alpha:=\nabla\rho^2_{111}\cdot\rho^{-3}_{11}$ is an affine invariant (see \cite{BGG16}) and we have
$\alpha_{\nabla}=4(\Gamma_{12}{}^2-(\Gamma_{12}{}^2)^2+\Gamma_{11}{}^2\Gamma_{22}{}^2)^{ {-1}}$ and
$\alpha_{\tilde{\nabla}}=\alpha_{\nabla}\cdot(a-e^{x^1})^2(a+e^{x^1})^{-2}$.
Since $\alpha_{\tilde{\nabla}}$ is non-constant for $a\ne0$, we are getting affine inequivalent surfaces which are strongly Liouville equivalent. If $a=0$, we obtain an isomorphic Type~$\mathcal{A}$
structure.
\end{example}

 \begin{example}\rm\label{E-A.3} 
We set $\Gamma_{11}{}^1=0$, $\Gamma_{12}{}^1=0$, $\Gamma_{22}{}^1=0$, as in Theorem \ref{TA.1}-(1.b). 
We then have $\rho_{\nabla}=\{\Gamma_{11}{}^2\Gamma_{22}{}^2-(\Gamma_{12}{}^2)^2\}dx^1\otimes dx^1$ and $\nabla\rho_{\nabla}=0$.
Since ${{x^1}}\in E(0,\nabla)$, we follow Theorem \ref{T1.3} and consider the strongly Liouville equivalent connection $\tilde\nabla$ determined by the 1-form $\omega=-d \log x^1$.
We verify that 
$$
\rho_{\tilde\nabla}=\rho_\nabla\quad\text{and}\quad
\tilde\nabla\rho_{\tilde\nabla}=4(x^1)^{-1}\rho_{\nabla,11}dx^1\otimes dx^1\otimes dx^1
$$ 
so this is not a symmetric space if we
choose $\Gamma_{11}{}^2\Gamma_{22}{}^2-(\Gamma_{12}{}^2)^2\neq 0$. Hence $\nabla$ is not locally isomorphic to $\tilde\nabla$. 
\end{example}

\subsubsection{Higher dimensional examples of Type~$\mathcal{A}$}\label{S-B}

Let $\mathcal{M}=(\mathbb{R}^3,\nabla)$ be a Type $\mathcal{A}$ geometry, so the Christoffel symbols $\Gamma_{ij}{}^k$ are constant.
We shall only list the non-zero Christoffel symbols in what follows and omit the details of the computation. Here $\mu_3=-\frac{1}{2}$.
The following is an example where $\rho$ is non-degenerate and $\mu\ne-\frac12$.

\begin{example}\label{EB.1}\rm 
Set the non-zero Christoffel symbols $\Gamma_{12}{}^3=1$, $\Gamma_{13}{}^1=3$, $\Gamma_{23}{}^2=4$, $\Gamma_{33}{}^3=5$. Then 
$\rho_{\nabla}=5dx^1\otimes dx^2+5dx^2\otimes dx^1+10dx^3\otimes dx^3$. We have
$$
E(\mu,\nabla)=\left\{\begin{array}{ll}
\operatorname{Span}\{e^{3x^3},x^1e^{3x^3}\}&\text{ if }\mu=-\frac35\\
\operatorname{Span}\{1\}&\text{ if }\mu=0\\
\{0\}&\text{ otherwise}\end{array}\right\}\,.
$$\end{example}

The Ricci tensor in the following example is degenerate, but non-zero, and there are an infinite number of non-trivial
eigenvalues; this is a genuinely new phenomena not present for Type~$\mathcal{A}$ surface models.
 
\begin{example}\label{EB.2}\rm
Let $\mathcal{M}_{x,y,z,w}:=(\mathbb{R}^3,\nabla)$ be a $3$-dimensional Type $\mathcal{A}$ model
where the (possibly) non-zero Christoffel symbols are:
$$
\Gamma_{11}{}^1=z,\,\,\,\,
\Gamma_{12}{}^1=1,\,\,\,\,
\Gamma_{13}{}^1=x,\,\,\,\,
\Gamma_{22}{}^2=1,\,\,\,\,
\Gamma_{23}{}^1=x,\,\,\,\,
\Gamma_{33}{}^2=y,\,\,\,\,
\Gamma_{33}{}^3=w.
$$
The Ricci tensor is 
$$
\rho_{\nabla}=\left(\begin{array}{ccc}0&0&0\\0&0&x(z-1)\\0&x(z-1)&wx-x^2+2y\end{array}\right)\,.
$$ 
Depending on the values of $x$, $y$, $z$ and $w$, $\dim\{E(-\frac12,\nabla)\}$ is as follows:
\begin{enumerate}
	\item $\dim\{E(-\frac12,\nabla)\}=0$ if and only if  $x\neq 0$  and either  $z=0$ or $z\notin\{0,1\}$ and $w\neq\frac{x+2xz-xz^2}{2z}$.
	\item $\dim\{E(-\frac12,\nabla)\}=1$ if and only if $x\neq 0$, $z\notin\{0,1\}$ and $w=\frac{x+2xz-xz^2}{2z}$.
	\item $\dim\{E(-\frac12,\nabla)\}=2$ if and only if $x\ne0$, $z=1$, $w\ne x$.
	\item $\dim\{E(-\frac12,\nabla)\}=4$ if and only if $x=0$ or $w=x$ and $z=1$.
\end{enumerate}

\end{example}

\subsection{Type~$\mathcal{B}$ surface models}\label{S-C}

Let $\mathcal{M}=(\mathbb{R}^+\times\mathbb{R},\nabla)$ be a Type~$\mathcal{B}$ affine surface model; 
the Christoffel symbols are given by $\Gamma_{ij}{}^k=(x^1)^{-1}C_{ij}{}^k$ where $C_{ij}{}^k$ are constant. We assume $\rho_{\nabla}\ne0$ to ensure the
geometry is not flat. We have $\mathcal{M}$ is also Type~$\mathcal{A}$ if and only if $(C_{12}{}^1,C_{22}{}^1,C_{22}{}^2)=(0,0,0)$;
the Ricci tensor has rank $1$ in this instance (see \cite{BGG16}). We first examine the Yamabe solitons, working modulo linear equivalence:

\begin{theorem}\label{th:Yamabe-type-B}
Let $\mathcal{M}$ be a Type $\mathcal{B}$ surface. Then $E(0,\nabla)=\operatorname{Span}\{1\}$ except in the following cases where we
also require $\rho_\nabla\ne0$.
\begin{enumerate}
\item $(C_{11}{}^1,C_{12}{}^1,C_{22}{}^1)=(-1,0,0)$, and $E(0,\nabla)=\operatorname{Span}\{1,\log(x^1)\}$.
 \smallbreak\item $(C_{11}{}^1,C_{12}{}^1,C_{22}{}^1)=\kappa(-1,0,0)$, 
$E(0,\nabla)=\operatorname{Span}\{1,(x^1)^{C_{11}{}^1+1}\}$, and \newline$\kappa\neq 1$.
 \smallbreak\item $(C_{11}{}^2,C_{12}{}^2,C_{22}{}^2)=(0,0,0)$, and $E(0,\nabla)=\operatorname{Span}\{1,x^2\}$.
 \smallbreak\item $(C_{11}{}^1,C_{12}{}^1,C_{22}{}^1)=c(C_{11}{}^2,C_{12}{}^2,C_{22}{}^2)$, and $E(0,\nabla)=\operatorname{Span}\{1,x^1-cx^2\}$.
\end{enumerate}
\end{theorem}

Any Type~$\mathcal{B}$ surface which is also Type~$\mathcal{A}$ is strongly projectively flat.
There are, however, strongly projectively flat surfaces of Type~$\mathcal{B}$ which are not of Type~$\mathcal{A}$.
Moreover, there exist Type $\mathcal{B}$ surfaces where $\dim\{E(-1,\nabla)\}=1$.

\begin{theorem}\label{TC.2}
Let $\mathcal{M}$ be a Type~$\mathcal{B}$ surface. Let $\mu=-1$. Then one of the following holds
\begin{enumerate}
\item $\dim\{E(-1,\nabla)\}=1$ if and only if $\mathcal{M}$ is
linearly equivalent to:
\begin{enumerate}
\item $C_{22}{}^1=0$, $C_{22}{}^2=C_{12}{}^1\neq 0$, or
\item $C_{22}{}^1=\pm 1$, $C_{12}{}^1=0$, $C_{22}{}^2=\pm 2C_{11}{}^2\neq 0$,
$C_{11}{}^1=1+2C_{12}{}^2\pm(C_{11}{}^2)^2$.
\end{enumerate}
\item $\dim\{E(-1,\nabla)\}=3$ if and only if $\mathcal{M}$ is strongly projectively flat. In this case $\mathcal{M}$ is
linearly equivalent to one of the surfaces:
\begin{enumerate}
\item $C_{12}{}^1=C_{22}{}^1=C_{22}{}^2=0$ (i.e. $\mathcal{M}$ is also of Type~$\mathcal{A}$).
\item $C_{11}{}^1=1+2C_{12}{}^2$, $C_{11}{}^2=0$, $C_{12}{}^1=0$, $C_{12}{}^2\neq 0$, $C_{22}{}^1=\pm1$, $C_{22}{}^2=0$. 
\end{enumerate}
\end{enumerate} 
\end{theorem}

Let $\mu\ne0$ and $\mu\ne-1$. In the Type~$\mathcal{A}$ setting, Theorem \ref{TA.1} shows that $\dim\{E(\mu,\nabla)\}=0$ or $\dim\{E(\mu,\nabla)\}=2$.
The situation is quite different in the Type~$\mathcal{B}$ setting as there are examples where $\dim\{E(\mu,\nabla)\}=1$.

\begin{theorem}\label{TC.3} 
Let $\mathcal{M}$ be a Type $\mathcal{B}$ model which is not of Type $\mathcal{A}$ with $\rho_{s,\nabla}\ne 0$ and let $\mu\ne 0,-1$.
\begin{enumerate}
\item $\dim\{E(\mu,\nabla)\}\geq 1$ if and only if $\mathcal{M}$ is linearly equivalent to a surface given by
$C_{22}{}^1=\pm 1$\,, $C_{12}{}^1=0$\,, $C_{22}{}^2=\pm 2C_{11}{}^2$\,,
where $\mu$ is determined by
$\mu=\Delta^{-2}\{1+2C_{12}{}^2\pm 2(C_{11}{}^2)^2 -(C_{11}{}^1-C_{12}{}^2)^2\}$, for $\Delta:=1-C_{11}{}^1+C_{12}{}^2\ne 0$.

\item $\dim\{E(\mu,\nabla)\}=2$ if and only if $\mathcal{M}$ is
linearly equivalent to one of the following two surfaces: 
\goodbreak\begin{enumerate}
\item $
C_{11}{}^1=-1+C_{12}{}^2$\,,
$C_{11}{}^2=0$\,,
$C_{12}{}^1=0$\,,
$C_{22}{}^1=\pm 1$\,,
$C_{22}{}^2=0$\,,
where $\mu=\frac{1}{2}C_{12}{}^2\neq 0$.
\item $
C_{11}{}^1=-\frac{1}{2}(5\pm 16 (C_{11}{}^2)^2)$\,,
$C_{12}{}^1=0$\,,
$C_{12}{}^2=-\frac{1}{2}(3\pm 8(C_{11}{}^2)^2)$\,,
$C_{22}{}^1=\pm 1$\,,
$C_{22}{}^2=\pm 2 C_{11}{}^2$\,,
where $\mu=-\frac{3\pm 8(C_{11}{}^2)^2}{4\pm 8(C_{11}{}^2)^2}$.
and where $C_{11}{}^2\neq 0,\pm\frac{1}{\sqrt{2}}$\,.
\end{enumerate}
\end{enumerate}
\end{theorem}

\begin{remark}\rm 
The existence of examples where
$\dim\{E(0,\nabla)\}$ is either $0$, $1$ or $2$ was shown in Theorem \ref{th:Yamabe-type-B}.
The surfaces of Theorem~\ref{TC.2} provide examples where one has $\dim\{E(-1,\nabla)\}=1$
and $\dim\{E(-1,\nabla)\}=3$. 
In Theorem~\ref{TC.3}, we gave examples of homogeneous affine surfaces where $\dim\{E(\mu,\nabla)\}=1$ and $\dim\{E(\mu,\nabla)\}=2$, for arbitrary $\mu\neq 0,-1$.
Thus all values for $\dim\{E(\mu,\nabla)\}$ are permissible.
\end{remark}

\end{document}